\documentclass[times]{nyjm}

\usepackage{tikz}

\newcommand{\bea}{\begin{eqnarray*}}
\newcommand{\eea}{\end{eqnarray*}}

\newcommand{\bm}{\begin{pmatrix}}
\newcommand{\fm}{\end{pmatrix}}
\newcommand{\bvm}{\begin{vmatrix}}
\newcommand{\fvm}{\end{vmatrix}}
\newcommand{\bbm}{\begin{bmatrix}}
\newcommand{\fbm}{\end{bmatrix}}

\newcommand\Z{\mathbb Z}

\newcommand\step{\textrm{step}}

\newcommand\V{\mathcal{V}}
\newcommand\W{\mathcal{W}}

\newcommand\m[1]{_{#1}}

\newcommand{\abs}[1]{\left | #1 \right |}

\DeclareMathOperator{\Ex}{\mathbb{E}}
\DeclareMathOperator{\Prob}{\mathbb{P}}
\DeclareMathOperator{\Var}{Var}

\begin{document}

\title{Random Nilpotent Groups of Maximal Step}

\author{ Phillip Harris}

\begin{abstract}
Let $G$ be a random torsion-free    nilpotent group generated by two random words of length $\ell$ in $U_n(\Z)$. Letting $\ell$ grow as a function of $n$, we analyze the step of $G$, which is bounded by the step of $U_n(\Z)$. We prove a conjecture of Delp, Dymarz, and Schafer-Cohen, that the threshold function for full step is $\ell = n^2$. 
\end{abstract}

\maketitle


A group $G$ is nilpotent if its lower central series 
\[
G = G_0 \geq G_1 \geq \dots \geq G_r = \{ 0 \}
\]
defined by $G_{i+1} = [G, G_i]$, eventually terminates. The first index $r$ for which $G_r = 0$ is called the \textit{step} of $G$. One may ask what a generic nilpotent group looks like, including its step. Questions about generic properties of groups can be answered with \textit{random groups}, first introduced by Gromov \cite{Ollivier}. Since Gromov's original \textit{few relators} and \textit{density} models are nilpotent with probability 0, they cannot tell us about generic properties of nilpotent groups. Thus there is a need for new random group models that are nilpotent by construction. 

Delp et al \cite{randnilp0} introduced a model for random nilpotent groups, motivated by the observation that any finitely generated torsion-free nilpotent group can be embedded in the group $U_n(\Z)$ of $n \times n$ upper triangular integer matrices with ones on the diagonal \cite{Hall}. Note that, since any finitely generated nilpotent group contains a torsion-free subgroup of finite index, we are not losing much by restricting our attention to torsion-free groups. (Another model is considered in \cite{randnilp1}).

We construct a random subgroup of $U_n(\Z)$ as follows. 
Let $E_{i,j}$ be the elementary matrix with 1's on the diagonal, a 1 at position $(i, j)$ and 0's elsewhere. 
Then $S = \{ E^{\pm 1}_{i,i+1} : 1 \leq i < n \}$
forms the standard generating set for $U_n(\Z)$. We call the entries at positions $(i,i+1)$ the \textit{superdiagonal} entries. Define a \textit{random walk} of length $\ell$ to be a product 
\[
V = V_1 V_2 \dots V_\ell 
\]
where each $V_i$ is chosen independently and uniformly from $S$. Let $V$ and $W$ be two independent random walks of length $\ell$. Then $G = \langle V, W \rangle$ is a random subgroup of $U_n(\Z)$. We have $\step(G) \leq \step(U_n(\Z))$, and it is not hard to check that $\step(U_n(\Z)) = n - 1$. If $\step(G) = n-1$ we say $G$ has \textit{full step}. 

Now let $n \to \infty$ and $\ell = \ell(n)$ grow as a function of $n$. We say a proposition $P$ holds \textit{asymptotically almost surely} (a.a.s.) if $\Prob [ P] \to 1$ as $n \to \infty$. Delp et al. gave results on the step of $G$, depending on the growth rate of $\ell$ with respect to $n$. 
\newtheorem{theorem}{Theorem}
\begin{theorem}[Delp-Dymarz-Schafer-Cohen]
Let $n, \ell(n) \to \infty$ and $G = \langle V , W \rangle$ where $V, W$ are independent random walks of length $\ell$. Then:
\begin{enumerate}
    \item If $\ell \in o(\sqrt{n})$ then a.a.s. $\step (G) = 1$, i.e. $G$ is abelian. 
    \item If $\ell \in o(n^2)$ then a.a.s. $\step (G) < n - 1$. 
    \item If $\ell \in \omega (n^3)$ then a.a.s. $\step (G) = n - 1$, i.e. $G$ has full step.
\end{enumerate}
\end{theorem}
In this paper we close the gap between cases 2 and 3. 
\begin{theorem}\label{thm:main}
Let $n, \ell(n) \to \infty$ and $G = \langle V , W \rangle$. If $\ell \in \omega(n^2)$ then a.a.s. $G$ has full step. 
\end{theorem}
To prove this requires a careful analysis of the nested commutators that generate $G_{n-1}$. In Section 1, we give a combinatorial criterion for a nested commutator of $V$'s and $W$'s to be nontrivial. In Section 2, we show this criterion is satisfied asymptotically almost surely when $V, W$ are random walks. 
\section{Nested Commutators}
Let $G = G_0 \geq G_1 \geq \dots$ be the lower central series of $G$. We have
\[
G_i = [G, G_{i-1}] = [G, [G, \dots, [G, G] \dots ]]
\]
In particular, $G_i$ includes all $i+1$-fold nested commutators of elements of $G$. We restrict our attention to commutators where each factor is $V$ or $W$. 

Let $\{ 0 , 1\}^d$ be the $d$-dimensional cube, or the set of all length $d$ binary vectors. For $x \in \{0,1\}^d, y \in \{0,1\}^{e}$ we define the norm $\abs{x} = \sum_{1 \leq i \leq d} x_i$ and the concatenation $xy \in \{0,1\}^{d + e}$. For example if $x = (1, 0, 0)$ and $y = (0, 1)$ then $xy = (1, 0, 0, 0, 1) = 10^31$.

We define a family of maps $C_d : \{ 0, 1 \}^d \to G_d$ as follows.
\begin{align}
C_1(1) &= V
\\
C_1(0) &= W
\\
C_{d}(1x) &= [V, C_{d-1}(x)]
\\
C_{d}(0x) &= [W, C_{d-1}(x)]
\end{align}
Thus for example $C_5(10^31) = C_5(10001) = [V, [W, [W, [W, V]]]]$. We omit the subscript $d$ when it is obvious. 
To prove $G$ has full step it suffices to find an $x \in \{0,1\}^{n-1}$ such that $C(x)$ is nontrivial. We begin with Lemma 2.3 from \cite{randnilp0}, which gives a recursive formula for the entries of a nested commutator. 
\newtheorem{lemma}{Lemma}
\begin{lemma}
Let $a \in \{0,1\}, x \in \{0,1\}^{d-1}$. Then $C(ax) \in G_d$ and we have
\begin{align}\label{eqn:Crec}
    C(ax)_{i,i+d}
    &=
    C(a)_{i,i+1}
    C(x)_{i+1,i+d}
    -
    C(a)_{i+d-1,i+d}
    C(x)_{i,i+d-1}
\end{align}
and furthermore $C(ax)_{i,j} = 0$ for $j < i+d$. 
\end{lemma}
In particular, for $d = n-1$ only the upper rightmost entry $C(ax)_{1,n}$ can be nonzero. From the formula it is clear that $C(ax)_{i,i+d}$ is a degree-$d$ polynomial in the superdiagonal entries of $V$ and $W$. Let us state this more precisely and analyze the coefficients of the polynomial. 

\begin{lemma}\label{lem:master}
Let $d \geq 1$. There exists a function $K_d : \{0,1\}^d \times \{0,1\}^d \to \Z$ such that for $1 \leq i \leq n - d$ we have
\begin{align}\label{eq:master}
    C(x)_{i, i+d} 
    &= 
    \sum_{
    \substack{y \in \{0,1\}^d \\ \abs{y} = \abs{x}}
    }
    K_d(x,y) \prod_{i \leq j < i + d} V_{j,j+1}^{y_j} W_{j,j+1}^{1- y_j}
\end{align}
Furthermore, setting $K_d(x,y) = 0$ for $\abs{x} \neq \abs{y}$ we have a recursion
\begin{align}\label{eq:krec}
    K_d(ax, byc) = K_1(a,b)K_{d-1}(x,yc) - K_1(a,c)K_{d-1}(x,by)
\end{align}
with base cases
\begin{align*}
    K_1(0,0) &= K_1(1,1) = 1
    \\
    K_1(0,1) &= K_1(1,0) = 0
\end{align*}
\end{lemma}
Note that $K_d(x,y)$ does not depend on $i$. We also drop the subscript $d$ since it can be inferred from $x$ and $y$. 
\begin{proof}
Abbreviate
\begin{align*}
    U(i,d,y) := \prod_{i \leq j < i+d} V_{j,j+1}^{y_j} W_{j,j+1}^{1 - y_j}
\end{align*}
We first prove inductively that there exist coefficients $K_d : \{0,1\}^d \times \{0,1\}^d \to \Z$ such that 
\begin{align*}
    C(x)_{i,i+d} 
    &=
    \sum_{y \in \{0,1\}^d}
    K_d(x,y)
    U(i,d,y)
\end{align*}
The case $d = 1$ is trivial. 
Assume it holds for $d-1$. Let $a \in \{0,1\}, x \in \{0,1\}^{d-1}$, then we have
\begin{align*}
    C(ax)_{i,i+d}
    &=
    C(a)_{i,i+1}C(x)_{i+1,i+d} - C(a)_{i+d-1,i+d}C(x)_{i,i+d-1}
\textbf{}\end{align*}

Expanding $C(a)_{i,i+1}$ and $C(x)_{i+1,i+d}$, the first term is 
    
\begin{align*}
    &=
    \left[ K_1(a,1)V_{i,i+1} + K_1(a,0)W_{i,i+1} \right]
    \left[ 
    \sum_{y \in \{0,1\}^{d-1}}
    K_{d-1}(x,y)
    U(i+1,d-1,y)
    \right]
    \\
    &=
    \sum_{y \in \{0,1\}^{d-1}}
    K_1(a,1)K_{d-1}(x,y) 
    U(i,d,1y)
    +
    K_1(a,0)K_{d-1}(x,y) 
    U(i,d,0y)
    \\
    &= 
    \sum_{\substack{
    b, c \in \{0, 1\}
    \\
    y' \in \{0, 1 \}^{d-2}
    }
    }
    K_1(a, b) K_{d-1}(x, y'c) 
    U(i,d,by'c)
\end{align*}
Similarly the second term is 
    \begin{align*}
    &= 
    \sum_{\substack{
    b, c \in \{0, 1\}
    \\
    y' \in \{0, 1 \}^{d-2}
    }
    }
    K_1(a, c) K_{d-1}(x, by')
    U(i,d,by'c)
\end{align*}
Combining we get
\begin{align*}
    C(ax)_{i,i+d}
    &=
    \sum_{
    \substack{
    b, c \in \{0, 1\}
    \\
    y \in \{0, 1\}^{d-2}
    }
    }
    \left[ 
    K_1(a,b)K_{d-1}(x,yc) 
    - 
    K_1(a,c)K_{d-1}(x,by)
    \right] 
    U(i,d,byc)
\end{align*}
And setting $K_d(ax,byc) = K_1(a,b)K_{d-1}(x,yc) - K_1(a,c)K_{d-1}(x,by)$ the lemma is proved for $d$. 
It is also easy to see inductively that $K_d(x,y) = 0$ for $\abs{x} \neq \abs{y}$, so we may add the condition $\abs{x} = \abs{y}$ under the sum to get Equation $\ref{eq:master}$.
\end{proof}
We now have a strategy for choosing $x \in \{0, 1\}^{n-1}$ such that $C(x)$ is nontrivial. In the random model, it may happen that $V_{i, i+1} = 0$ for some $i$. Define the vector $v \in \{0, 1\}^{n-1}$ by $v_i = 1$ if $V_{i,i+1} \neq 0$ and $v_i = 0$ otherwise. For now assume $0 < \abs{v} < n-1$. If we choose $x$ such that $\abs{x} = \abs{v}$, then Equation \ref{eq:master} simplifies to
\begin{equation}
    C_{n-1}(x)_{1, n} 
    =
    K_d(x,v)
    \prod_{1 \leq i < n} V_{i,i+1}^{v_i} W_{i,i+1}^{1 - v_i}
\end{equation}
If we assume there is no $i$ such that $V_{i,i+1} = W_{i,i+1} = 0$, the product of matrix entries is nonzero. So we just need to choose $x$ such that $K_d(x,v) \neq 0$. 
We can do this with some additional conditions on $v$. 
\begin{lemma}\label{lem:finish}
Let $v \in \{ 0, 1 \}^{n-1}$ with $0 < \abs{v} < n - 1$. Write $v = 1^{a_1} 0 1^{a_2} \dots 1^{a_k-1} 0 1^{a_k}$. Assume that $a_i \geq 1$ for all $i$, i.e., there are no adjacent 0's, and that $a_1 \neq a_k$. Then there exists $x \in \{ 0, 1\}^{n-1}$ such that $K(x,v) \neq 0$. 
\end{lemma}
We will prove in section 2 that all assumptions used hold asymptotically almost surely. 

\begin{proof}
Using Equation \ref{eq:krec}, the following identities are easily verified by induction:
\begin{enumerate}
    \item If $a, b \geq 0$, then \[ K(1^{a+b}0,1^a01^b) = {a + b \choose a}(-1)^b
    \]
    \item If $a, b \geq 1, c \geq 0$ with $c < \min (a,b)$, then
    \[ 
    K(1^c0x, 1^a y 1^b) = 0
    \]
    \item Let $a, b \geq 0$. If $a < b$ then 
    \[
    K(1^a0x,1^a0y1^b) = K(x,y1^b)
    \] If $b < a$ then 
    \[ 
    K(1^a0x,1^ay01^b) = K(x, 1^ay)
    \]
    \item If $a, b \geq 0$ then 
    \[
    K(1^{a+b} 0^2 x, 1^a01y101^b) = 2{a + b \choose a} (-1)^b K(x, 1y1)
    \]
\end{enumerate}
Let $v = 1^{a_1} 0 1^{a_2} \dots 0 1^{a_k}$. First assume $k = 2 \ell$ is even. Applying identity 4 repeatedly we reduce to the case $v = 1^{a_\ell} 0 1^{a_{\ell + 1}}$, then apply identity 1. Explicitly we have
\begin{align}
    x 
    &=
    1^{a_1 + a_{2 \ell}} 0^2 1^{a_2 + a_{2\ell - 1}} 0^2 \dots 1^{a_\ell + a_{\ell + 1}}0
    \\
    K(x,v) 
    &=
    2^\ell
    (-1)^{a_{2 \ell} + a_{2 \ell-1} + \dots + a_{\ell + 1}} 
    {a_1 + a_{2 \ell+1} \choose a_1}
    {a_2 + a_{2 \ell} \choose a_2}
    \dots 
    {a_\ell + a_{\ell + 1} \choose a_\ell}
\end{align}
If $k$ is odd, apply identity 3 once and proceed as before. 
\end{proof}

\section{Asymptotics}
In Section 1 we derived a combinatorial condition on the superdiagonal entries of $V$ and $W$ sufficient for $G$ to have full step. Define
\begin{align*}
    \mathcal{V} &= \{ i : 1 \leq i < n, V_{i,i+1} = 0 \}
    \\
    \mathcal{W} &= \{ i : 1 \leq i < n, W_{i,i+1} = 0 \}
\end{align*}
Then to apply Lemma \ref{lem:finish} we need that
\begin{enumerate}
    \item $\V$ and $\W$ are nonempty.
    \item $\V \cap \W = \emptyset$.
    \item $\V$ has no adjacent elements.
    \item $\min \V \neq n - \max \V$.
\end{enumerate}
If condition (1) does not hold, then Theorem \ref{thm:main} follows by a modification of Lemma 5.4 in \cite{randnilp0}. 
We now show that in the random model, the superdiagonal entries satisfy conditions (2)-(4) asymptotically almost surely. Recall that $V$ and $W$ are random walks
\begin{align*}
    V &= V_1 V_2 \dots V_\ell
    \\
    W &= W_1 W_2 \dots W_\ell
\end{align*}
where each $V_i, W_i$ is chosen independently and uniformly from the generating set
$S = \{ E^{\pm 1}_{i,i+1} : 1 \leq i < n \}$
.
Define
\begin{align}
    \sigma_j(Z) = \begin{cases}
    1 & \text{ if $Z = E_{j,j+1}$}
    \\
    -1 & \text{ if $Z = E_{j,j+1}^{-1}$}
    \\
    0 & \text{ otherwise}
    \end{cases}
\end{align}
Then we have
\begin{equation}\label{eq:sigma}
    V\m{i,i+1} = \sum_{j=1}^\ell \sigma_i(V_j)
\end{equation}
When $\ell \gg n$, the superdiagonal entries $V_{i,i+1}$ behave roughly like independent random walks on $\Z$. 
We restate Corollary 3.2 from \cite{randnilp0}. 
\begin{lemma}\label{lem:prob}
Suppose $\ell = \omega(n)$. Then uniformly for $(k_1, \dots, k_d) \in \Z^d$ we have
\[
\Prob [ k_i \in \V \text{ for all $i$}] 
\sim 
\left(\frac{n}{2 \pi \ell}\right)^{d/2}
\]
\end{lemma}
Since $\V$ and $\W$ are i.d.d, we have $\Prob [ i \in \V \cap \W ] \ll n / \ell$, so by the union bound we have $\Prob [ \V \cap \W \neq \emptyset] \ll n^2/\ell \to 0$. Thus condition (2) holds a.a.s. For conditions (3) and (4) we will need a bound on the size of $\V$.

\begin{lemma}
Fix $\epsilon > 0$. Then $\Prob [ \abs{\V} > \epsilon \sqrt{n} ] \to 0$ as $n \to \infty$. 
\end{lemma}
\begin{proof}
Define random variables
\begin{align*}
X_i &= \begin{cases}
1 & V(i,i+1) = 0
\\
0 & V(i,i+1) \neq 0
\end{cases}
\end{align*}
So $\abs{\V} = \sum_i X_i$. From 
Lemma \ref{lem:prob} we have $\Ex [ X_i ] \ll \sqrt{n/\ell}$ and $\Ex [ X_i X_j] \ll n/\ell$ for $1 \leq i < j < n$.
Hence $\Ex [ \abs{\V} ] \ll \sqrt{n^3/\ell}$ and $\Var [ \abs{\V} ] \ll n^3/\ell$.
By Chebyshev's inequality
\begin{align*}
    \Prob [ \abs{\V} \geq \epsilon \sqrt{n} ] 
    &\leq
    \Prob \left[ \abs{\V} - \sqrt{n^3/\ell} \geq \sqrt{n}(\epsilon - \sqrt{n^2/\ell}) \right]
    \\
    &\leq 
    \frac{1}{(\epsilon-\sqrt{n^2/\ell})^2(\ell/n^2)} 
    \to 0
\end{align*}
\end{proof}
Observe that the distribution of $\V$ is invariant under permutation. In other words, for a fixed set $\mathcal{S} \subset \{1, \dots, n-1 \}$ and a permutation $\pi$ on $\{1, \dots, n-1 \}$ we have
\begin{align*}
    \Prob [ \V = \mathcal{S}] = \Prob [ \V = \pi \mathcal{S}]
\end{align*}
and hence
\begin{align*}
    \Prob [ \V = \mathcal{S}] = \frac{1}{{n-1 \choose \abs{\mathcal{S}}}}\Prob [\abs{V} = \abs{\mathcal{S}}]
\end{align*}
Let $A(k)$ be the number of sets $\mathcal{S} \subset \{1, \dots, n-1 \}$ of size $k$ with at least one pair of adjacent elements. We have
\[
A(k) \leq (n-2) {n-3 \choose k - 2}
\]
Let $B(k)$ be the number of sets $\mathcal{S}$ for which $\min \mathcal{S} = n - \max \mathcal{S}$. Summing over the possible values of $\min \mathcal{S}$ we have 
\[
B(k) \leq \sum_{1 \leq a \leq n/2} {n - 1 - 2a \choose k - 2} 
\]
One easily checks
\[
\frac{A(k) + B(k)}{{n-1 \choose k}} \leq \frac{2k^2}{n}
\]
For $k \leq \epsilon \sqrt{n}$ this is $\leq 2 \epsilon^2$. On the other hand $\Prob [ \abs{V} > \epsilon \sqrt{n} ] \to 0$, so we are done. 

\textit{Acknowledgements.}
 We thank Tullia Dymarz for suggesting this problem and for many helpful discussions.







\end{document}